%% file: main.tex
\documentclass{amsart}

\usepackage[utf8]{inputenc}
\usepackage{caption}
\usepackage{enumerate}
\usepackage{amsmath}
\usepackage{amssymb}
\usepackage{ upgreek }
\usepackage{mathtools}
\usepackage{centernot}
\usepackage{stmaryrd}
\usepackage{esvect}
\usepackage{amsthm}
\usepackage{upgreek}
\usepackage{comment}
\usepackage{amsfonts}
\usepackage{physics}
\usepackage{tikz-cd}
\usepackage[utf8]{inputenc}
\usepackage[english]{babel}
\usepackage{graphicx}
\usepackage[margin=1in]{geometry}
\usepackage{mathtools}

\graphicspath{{images/}}

\renewcommand{\part}[1]{\noindent\textbf{Part #1)}}

\newcommand{\ba}{\begin{align*}}

\newcommand{\ea}{\end{align*}}

\newcommand{\Int}[1]{%
  {\kern0pt#1}^{\mathrm{o}}%
}

\newcommand{\vol}{\mathrm{vol}}

\newcommand{\bp}{\begin{pmatrix}}
\newcommand{\ep}{\end{pmatrix}}

\newcommand{\interior}[1]{%
  {\kern0pt#1}^{\mathrm{o}}%
}

\usepackage{xargs}                      
\usepackage{xcolor}
\usepackage[colorinlistoftodos,prependcaption,textsize=tiny]{todonotes}
\newcommandx{\kate}[2][1=]{\todo[linecolor=red,backgroundcolor=red!25,bordercolor=red,#1]{#2}}
\newcommandx{\zhen}[2][1=]{\todo[linecolor=lime,backgroundcolor=lime!25,bordercolor=lime,#1]{#2}}
\newcommandx{\yi}[2][1=]{\todo[linecolor=black,backgroundcolor=black!25,bordercolor=black,#1]{#2}}
\newcommandx{\kabir}[2][1=]{\todo[linecolor=purple,backgroundcolor=purple!25,bordercolor=purple,#1]{#2}}

\newcommandx{\nat}[2][1=]{\todo[linecolor=blue,backgroundcolor=blue!25,bordercolor=blue,#1]{#2}}
\newcommandx{\jonah}[2][1=]{\todo[linecolor=yellow,backgroundcolor=yellow!25,bordercolor=yellow,#1]{#2}}
\newcommandx{\colin}[2][1=]{\todo[linecolor=gray,backgroundcolor=gray!25,bordercolor=gray,#1]{#2}}
\begin{document}
\author[C. Adams, O. Eisenberg, J. Greenberg, K. Kapoor, Z. Liang,
K. O'Connor, N. Pacheco-Tallaj, Y. Wang] {Colin Adams, Or Eisenberg, Jonah Greenberg, Kabir Kapoor, Zhen Liang, \\
Kate O'Connor, Natalia Pacheco-Tallaj, Yi Wang}
\title{Turaev hyperbolicity of classical and virtual knots}
\begin{abstract}By work of W. Thurston, knots and links in the 3-sphere are known to either be torus links, or to contain an essential torus in their complement, or to be hyperbolic, in which case a unique hyperbolic volume can be calculated for their complement. We employ a construction of Turaev to associate a family of hyperbolic 3-manifolds of finite volume to any classical or virtual link, even if non-hyperbolic. These are in turn used to define the \emph{Turaev volume} of a link, which is the minimal volume among all the hyperbolic 3-manifolds associated via this Turaev construction. In the case of a classical link, we can also define the \emph{classical Turaev volume}, which is the minimal volume among all the hyperbolic 3-manifolds associated via this Turaev construction for the classical projections only.   We then investigate these new invariants.\end{abstract}

\maketitle

\theoremstyle{definition}
\newtheorem{definition}{Definition}[section]
\newtheorem{theorem}{Theorem}[section]
\newtheorem{lemma}[theorem]{Lemma}
\newtheorem{conjecture}{Conjecture}
\newtheorem{corollary}[theorem]{Corollary}
\newtheorem{proposition}[theorem]{Proposition}
\newtheorem{example}{Example}[section]
\newtheorem{question}{Question}
\newtheorem{remark}{Remark}


\maketitle

\section{Introduction}
\input{intro.tex}

\section{Orientability of Turaev surfaces}

\input{orientability.tex}

\section{Turaev volume invariant for classical and virtual links}
\label{turaevvolume}
\input{turaevvolume.tex}



\section{Some Turaev volumes}
\input{somevolumes.tex}
\input{unknot-diagrams.tex}

\subsection*{Acknowledgments}  
The authors are grateful for support they received from NSF Grant DMS-1659037 and the Williams College SMALL REU program. And thanks to Hans Boden for providing useful input.

\bibliography{mybib}{}
\nocite{*}
\bibliographystyle{plain}

\end{document}

%% file: intro.tex


The theory of links in $S^3$ (which here will be called classical links) is equivalent to the theory of classical links in $S^2\times I$, since removing the interiors of two balls from $S^3$ does not impact the theory. A natural extension is to the theory of links in $S\times I$ where $S$ is a compact, orientable surface of higher genus. It turns out that this extension occurs naturally through the theory of virtual knots and links, which originated from  Gauss codes. 


A Gauss code is a sequence of symbols that encodes a knot projection up to planar isotopy. It is obtained by starting at some point on an oriented knot diagram and travelling along the knot, writing down a sequence of integers, two copies of each corresponding to the crossings, letters O and U corresponding to whether we are passing over or under the crossing and $\pm$ signs representing the writhe of the crossing. In the case of an oriented link, we generate a finite collection of such sequences, one for each component. (N.B.: when we talk of Gauss codes, we always mean \textit{oriented} or \textit{signed} Gauss codes). While every classical knot diagram (up to planar isotopy) corresponds to a unique Gauss code (up to relabeling and cyclic permutation), the converse is not true--there are Gauss codes that do not correspond to any classical diagram. Generating knots from the ``missing'' Gauss codes was one of the primary motivations for virtual knot theory. The theory of the set of all Gauss codes modulo the corresponding equivalence relation induced by the Reidemeister moves is equivalent to that of virtual diagrams modulo classical and virtual Reidemeister moves. 

In \cite{KK} and \cite{virtual-equivalent-to-linkinthickenedsurface}, it was shown that in fact both of these theories are equivalent to the theory of links in thickened oriented surfaces $S \times I$ modulo ambient isotopy, stabilization, and destabilization, where by stabilization and destabilization we mean the addition and removal of empty handles. Further, when we project the link $L$ in $S \times I$ to $S \times \{1/2\}$ we obtain a surface-link pair $(S, D)$ where $D$ is a link diagram on $S$. In \cite{Kuperberg}, it was shown that the minimal genus realization of a link is uniquely determined up to isotopy.






Thus, we have the following equivalence:

\begin{theorem}
The following sets are in natural bijection.
    \begin{enumerate}[(i)]
        \item Virtual link diagrams in the plane modulo classical and virtual Reidemeister moves.
        \item Links in thickened surfaces modulo ambient isotopy, homeomorphisms, stabilizations, and destabilizations. 
        \item Surface-link pairs $(S,L)$ modulo isotopy, classical Reidemeister moves on the surface, handle attachments, and handle removals.
        \item (Virtual) Gauss codes modulo rewrites corresponding to Reidemeister moves. 
        \end{enumerate}
        \end{theorem}
        Note that the set of classical links is naturally included in the set of virtual links, which is to say that classical link diagrams which cannot be related by a sequence of classical Reidemeister moves cannot be related by a sequence of generalized Reidemeister moves. We refer the reader to \cite{kauffman-intro-to-virtual-knots} for a thorough exposition of virtual knots.
        
        \medskip

In 1978 (c.f.\cite{Thurston}), Thurston showed that a classical knot is either a torus knot, a satellite knot, or a hyperbolic knot. Similarly, a classical non-splittable link that does not contain an essential torus or annulus is hyperbolic. In \cite{virtualknotssummer}, hyperbolic invariants are extended to the virtual category by utilizing the equivalence of virtual links to links in thickened surfaces. 

\begin{definition} Let $S$ be a closed orientable surface. A link $L$ in $S \times I$ is {\it tg-hyperbolic} if:

\begin{enumerate}[(i)]
    \item When $S$ is a sphere, and the two spherical boundaries are capped off with balls, the complement of $L$ is hyperbolic.
    \item When $S$ is a torus, and the two torus boundaries are removed, the link complement is hyperbolic.
    \item When $S$ is neither a sphere nor torus, there exists a hyperbolic structure on the complement of $L$ in $S \times I$ such that the two boundaries are totally geodesic.
\end{enumerate}.
\end{definition}

If  the link in $S \times I$ is  tg-hyperbolic, we can associate a unique hyperbolic volume to it. We can also consider other hyperbolic invariants of the pair $(S \times I, L)$. Thus, we can define hyperbolic invariants of the original virtual link accordingly. See \cite{virtualknotssummer} for more on this, including a table of volumes of virtual knots of four or fewer classical crossings.

But in  both the classical and virtual categories, there exist knots and links such that their associated surface-link pair is not tg-hyperbolic and thus, to which these hyperbolic invariants do not apply. In this paper,  using the theory of Turaev surfaces, we extend hyperbolic invariants to every classical and virtual knot and link. 


\begin{definition}
    Given a connected classical or virtual link diagram $D$ in the plane, define the \emph{Turaev surface-link pair}, $\big(S_T(D), L_T(D)\big)$, to be the surface-link pair constructed as follows. Begin by embedding each crossing in a small disk in the plane. Then glue bands connecting adjacent classical crossings, ignoring any virtual crossings in between (that is, allowing one of the two bands involved  to pass over or under the other, it does not matter which), and adding a half-twist in the band if both endpoints are overcrossings or both are undercrossings, twisted in the direction shown in Figure \ref{Turaevribbon}. Then cap off each boundary component with a disk. Given a virtual link diagram $D$ and a surface link pair $(S',D')$ such that $(S',D') = (S_T(D),L_T(D))$, we say that $(S',D')$ is the \it{Turaev realization} of $D$.
\end{definition}

\begin{figure}[htbp]
    \captionsetup{width=.5\textwidth}
    \centering
    \includegraphics[scale=0.17]{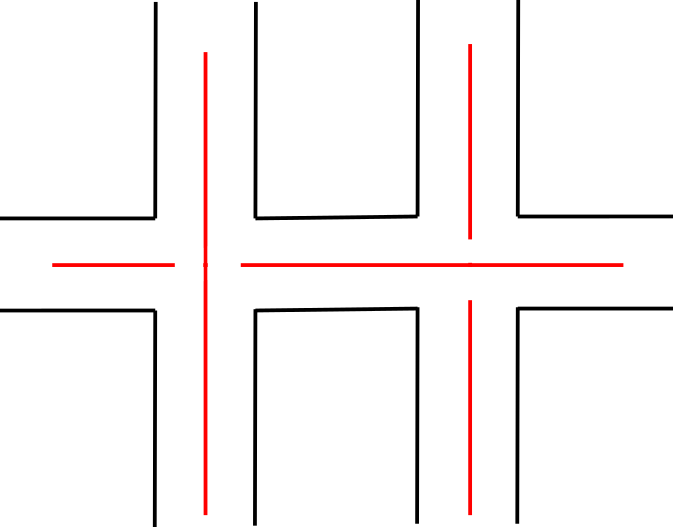}
    \hspace{1em}
    \includegraphics[scale=0.17]{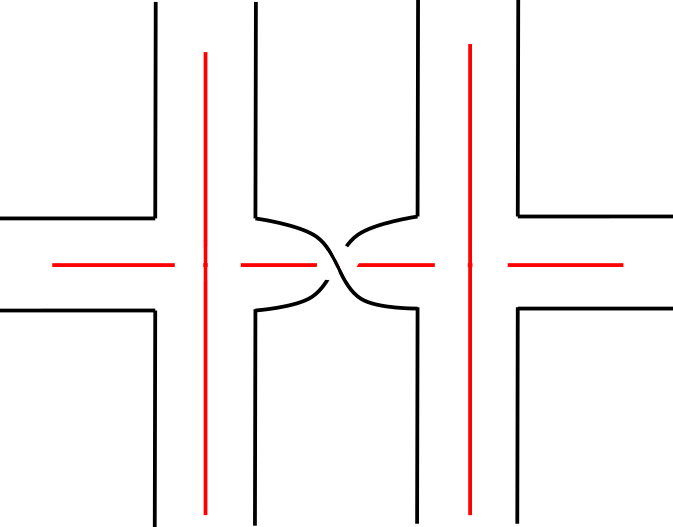}
    \hspace{1em}
    \includegraphics[scale=0.17]{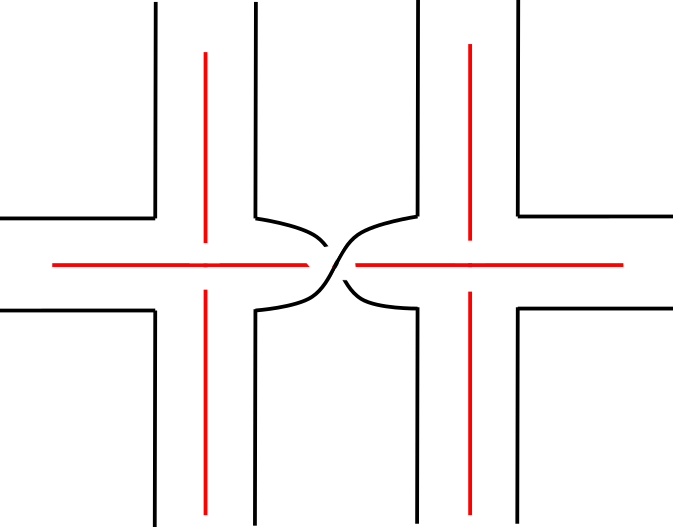}
    \caption{Turaev Realizations}
    \label{Turaevribbon}
\end{figure}

\noindent In the classical case, Turaev surfaces were introduced by Turaev in \cite{Turaev}.
See also \cite{Fomenko} and \cite{virtualknotsbook} where they appear independently as atoms. 

In the case of a classical knot, a second, equivalent construction of the Turaev surface-link pair, is as a cobordism between the $A$ state and the $B$ state of a knot, where the boundary components are then capped off with disks. The $A$ state of a knot is a collection of circles obtained by resolving each crossing with the $A$ smoothing in Figure~\ref{fig:absmoothings}. 
The $B$ state is defined similarly. See \cite{Turaev}. This construction of the Turaev surface-link pair is a special case of a \emph{state surface}, in which the cobordism is  between two opposite resolutions of the knot with some mix of $A$ and $B$ smoothings. 

\begin{figure}[h]
    \centering
    \includegraphics[width=0.35\textwidth]{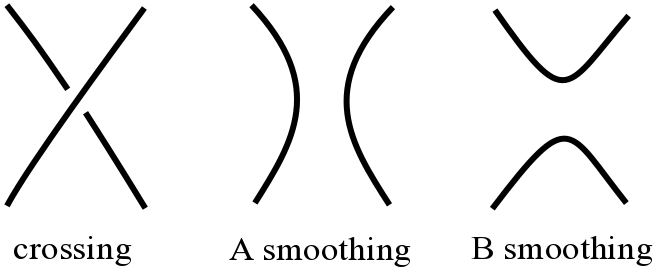}
    \caption{Smoothings of a crossing}
    \label{fig:absmoothings}
\end{figure}

The Turaev surface comes to us naturally with a link projection on it. The ribbon graph is precisely a normal neighborhood of this link projection in the Turaev surface. 

For classical knots, the Turaev surface is always orientable. The cobordism construction comes to us with a natural height function $h$, which is Morse, such that its only critical set $h^{-1}(0)$ is the knot projection, such that $h^{-1}(-1)$ is the all $A$-state smoothing, and such that $h^{-1}(1)$ is the all $B$-state smoothing. This demonstrates that the boundary components can be capped off to yield a closed surface embedded in $3$-space. 

However, in the virtual category, it is possible to obtain nonorientable Turaev surfaces. In this case, we consider $D'$ to be the link in the oriented twisted $I$-bundle over the nonorientable Turaev surface of $D$. See the discussion of atoms in \cite{virtualknotsbook} for another perspective on this.  

From here on, when we discuss a thickened surface $M$, we mean $S\times I$ when $S$ is orientable and the twisted $I$-bundle over $S$ when $S$ is nonorientable. Knots and links in oriented twisted $I$-bundles over non-orientable surfaces have previously been considered in the case of the projective plane in \cite{Dro} and more gnerally in \cite{Bou}. 

\begin{definition} The \it{Turaev genus}  $g_T(L)$ of a classical or virtual link $L$ is the minimal genus of a Turaev surface for $L$ over all the projections of $L$. In the case of an orientable surface, we use the usual genus. In the case of a nonorientable surface, we use one half the number of projective planes, the connected sum of which make up the Turaev surface. 
\end{definition}

Notice that this allows for half-integer genera. But the need for this definition of genus for nonorientable surfaces comes from the desire to be able to compare genera between orientable and non-orientable surfaces.

Of critical importance is the fact that the introduction of the half-twists in the ribbon graph has the effect of guaranteeing that the resulting link is alternating on the surface. By construction, every connected link diagram has an alternating Turaev realization. Furthermore, since boundary components are capped off with disks, every link diagram is realized as a \emph{fully alternating} link in its Turaev surface, in the sense defined in \cite{small2017}.

\begin{definition}
    A link projection $L$ on a closed surface $S$ is \emph{fully alternating} if it is alternating where the interior of every complementary region is an open disk.
\end{definition}

In the literature, a projection with disks as the complementary regions is sometimes called a \emph{cellular embedding}. Note that the projection must be connected to satisfy the criteria for being fully alternating.

\begin{definition}
    A link $L$ embedded in a thickened surface $M$ is \emph{prime} if there is no ball $B$ embedded in $M$ such that $\partial B$ intersects $L$ twice and $B$ contains some subset of $L$ other than an  unknotted arc.
\end{definition}

\begin{definition}
    A link diagram is \emph{Turaev prime} if it is realized as a prime link in its Turaev surface.
\end{definition}

Although we have already defined what it means for the complement of a link $L$ in $S \times I$  to be tg-hyperbolic when $S$ is orientable, we need to extend this to the case of $S$ nonorientable.

\begin{definition} Let $S$ be a closed nonorientable surface. A link $L$ in the twisted I-bundle $M$ over $S$  is {\it tg-hyperbolic} if:

\begin{enumerate}[(i)]
    \item When $S$ is a projective plane, and the spherical boundary of $M$ is capped off with a ball, the complement of $L$ is hyperbolic.
    \item When $S$ is a Klein bottle, and the torus boundary of $M$ is removed, the link complement is hyperbolic.
    \item When $S$ is neither a projective plane nor a Klein bottle, there exists a hyperbolic structure on the complement of $L$ in $M$ such that the boundary of $M$ is totally geodesic.
\end{enumerate}
\end{definition}


In \cite{small2017}, two relevant results were proved, both extensions of Menasco's results for links in the 3-sphere (\cite {Menasco}). Lemma 14 of \cite{small2017}, with a straightforward extension that we provide at the beginning of Section 3 states:  

\begin{theorem}\label{alttg}
    Let $S$ be a closed surface. Then a prime fully alternating link $L$ in a thickened surface $M$ over $S$ is $tg$-hyperbolic, except when:
    
    \begin{enumerate}[(i)]
        \item  $S$ is a sphere  and $L$ is a 2-braid.  
        \item $S$ is the projective plane and $L$ is the analog of a 2-braid. 
        \item $S$ is the projective plane and there exists a simple closed curve that intersects the projection transversely once.
    \end{enumerate}
\end{theorem}

We say a link projection is {\it reduced} if there are no monogonal disk faces in its complement. Note that any link projection on any closed surface is equivalent to one that is reduced.  

\begin{theorem}\label{primefullyalt} A fully alternating link  $L$ in a thickened surface $M$ is prime if and only if there is no disk $E$ on the projection surface such that $\partial E$ intersects a reduced fully alternating projection of $L$ transversely at two points and such that there exist crossings in $E$.
\end{theorem}

We call a reduced fully alternating link projection that has no such disks an {\it obviously prime projection}. Hence, the import of the theorem is that a fully alternating link in a thickened surface is prime if and only if a reduced fully alternating projection of it is obviously prime.

In Section 3, we provide necessary and sufficient conditions for a link diagram be realized as a prime, fully alternating link in its Turaev surface. We use this to show that every link, classical or virtual, has some diagram that generates a prime fully alternating link in $M$, which is therefore $tg$-hyperbolic and thus, has a unique hyperbolic volume associated to it. In fact, there are infinitely many such diagrams for the link. 

\begin{definition}
    Let $D$ be a Turaev prime diagram of a link $L$. The \emph{Turaev volume} of $D$, denoted $\vol_T(D)$, is the hyperbolic volume of the Turaev realization of $D$. 
\end{definition}

\begin{definition}
    \label{def:turaevvolume}
    For a link $L$ that may be classical or virtual, the \emph{Turaev volume} of $L$, denoted $\vol_{T}(L)$ is the minimum over all Turaev prime diagrams $D$ of $L$ of $\vol_T(D)$. For a classical link $L$, the \emph{classical Turaev volume}, denoted $\vol_{CT}(L)$, is taken to be the minimum over all classical diagrams of $L$. Since the set of hyperbolic volumes is well-ordered (see \cite{Thurston}), there is always a minimum.
\end{definition}

It is unknown at this time whether a hyperbolic alternating classical link has hyperbolic volume equal to its classical Turaev volume, but it is known in certain cases (see Examples 5 and 6 of Section 4), and we conjecture this to be the case. Moreover, in this case, we expect the classical Turaev volume equals the Turaev volume. In general, it is unclear if Turaev volume equals classical Turaev volume for nontrivial classical links, although we conjecture they are distinct for the trivial knot. See Example 8 of Section 4.

Since every link has infinitely many Turaev volumes associated to it, we can also define an invariant called the \emph{Turaev spectrum}.

\begin{definition}
    \label{def:turaevspectrum}
    For a link $L$ that may be classical or virtual, the \emph{Turaev spectrum} of $L$ is the ordered sequence of all Turaev volumes of diagrams of $L$. For a classical link $L$, the classical Turaev spectrum is the ordered sequence of all Turaev volumes corresponding to classical projections of $L$.
\end{definition}

Although we will not pursue it here, one could also consider the Turaev spectrum for each fixed genus.

In this paper, we show that these invariants are well-defined. Section 2 provides a discussion of the determination of orientability or nonorientability of Turaev surfaces. In Section 3, we prove that every knot and link, virtual or classical, has a diagram such that its Turaev realization is hyperbolic and that every classical knot or link has a classical diagram such its Turaev realization is hyperbolic, thereby making Turaev volume well-defined for all knots and links and classical Turaev volume well-defined for classical knots and links.  In Section 4, we provide explicit examples of classical Turaev volumes and some conjectures about both Turaev and classical Turaev volumes. At this time, no explicit Turaev volumes are known.

Throughout we use $v_{oct} \approx 3.6638 \dots$ to represent the volume of an ideal regular hyperbolic octahedron.

%% file: orientability.tex
While all classical links have orientable Turaev surfaces corresponding to any classical projection, every link has projections that generate nonorientable Turaev surfaces.
In this section, we present necessary and sufficient conditions for orientability of the Turaev surface.  We begin with some definitions related to Gauss codes.

\begin{definition}
A \emph{generalized Gauss code} is a Gauss code where it is no longer required that each pair of entries corresponding to a given crossing must have one O and one U associated to them. A generalized Gauss code is \emph{orientable} if the pair of appearances of an entry does have one O and one U. Otherwise call the generalized Gauss code \emph{nonorientable}. 
\end{definition}

\begin{example}
$$1O^+2O^-1O^+2U^-$$ is a valid generalized Gauss code, though it is not a valid Gauss code since both occurrences of the crossing $1$ are over crossings. 
\end{example}

\begin{definition}
Given a Gauss code, define the \emph{Turaev code} associated to that code to be the generalized Gauss code obtained by changing the O and U labels to make the code alternating.
\end{definition}

\begin{example}\label{tur_exmp} The Turaev code associated to the Gauss code $$1O^+2O^-3U^+1U^+4O^-2U^-3O^+4U^-$$ is the alternating code $$1O^+2U^-3O^+1U^+4O^-2U^-3O^+4U^-.$$
\end{example}

Note that the Turaev code of a given Gauss code is orientable if and only if the number of entries between the two occurences of each crossing number in the code is even. A Gauss code satisfying this property is called {\it alternatable} in \cite{Kamada} and \cite{Viro}.

For instance, the code in Example \ref{tur_exmp} is nonorientable since the sequence $1U^+4O^-2U^-$ occurs between the two instances of crossing number 3. Also, note that there is an ambiguity in whether the Turaev code of a given Gauss code begins with an over crossing or an under crossing. There is no canonical choice for this. Nevertheless, this ambiguity is essentially the ambiguity in the orientation of the associated Turaev surface. Since the associated link living in a thickened surface is independent of a choice of orientation, though, this is not an issue for our purposes.


\noindent 

\begin{theorem}
\label{orientabilityTuraev}
   For a virtual knot diagram $D$, the following are equivalent.
 \begin{enumerate}[(i)]
        \item The Turaev code corresponding to $D$ is nonorientable.
       \item The Gauss code $G$ for $D$ has some crossing $i$ such that the two entries for $i$ are separated by an odd number of Gauss code entries (not counting the $i$ entries themselves).  
       \item The diagram $D$ has an associated Turaev surface which is nonorientable.
       \end{enumerate}
      
 \end{theorem}
 
 A similar result applies for links of two or more components but then, either there exists a component such that  both copies of a crossing number appear in the Gauss code of the component and separate off an odd number of entries, or there exists a component such that its Gauss code  has odd length.

\begin{proof} As mentioned previously, the equivalence between (i) and (ii) is immediate. Then (i) implies (iii) since a nonorientable Turaev code has a crossing that does not have a U on one appearance in the code and an O on the other. Hence the path along the knot between the two appearances must pass through an odd number of half-twists in order that the resultant knot on the Turaev surface be alternating. The path then has neighborhood on the surface that is a M\"obius band.

To prove that (iii) implies (ii) is more involved. Treating the Gauss code as cyclical, we prove the contrapositive.  Suppose that all pairs of entries corresponding to a crossing separate an even number of entries in the Gauss code but that the surface $S$ is nonorientable. Note that since the total number of entries in the Gauss code $G$ is even, if a pair separates an even number of entries to one side, it must separate an even number to the other side as well. Since $S$ is nonorientable, there must be a closed path through the graph of the ribbon-surface that is nonorientable. We take a minimal such path $\gamma$. In other words, $\gamma$ passes through each edge and vertex at most once and it passes through an odd number of half-twists in the ribbon surface. In particular, at each vertex of the graph, $\gamma$ must either turn to the right or left or go straight through the vertex. Let $a_1, a_2,\dots, a_n$ be the crossing labels on the vertices where $\gamma$ turns. Note they are all distinct. 

We count how many crossings $\gamma$ passes straight through between the vertices corresponding to $a_i$ and $a_{i+1}$ and call that $b_i$. Then the fact the neighborhood of the path $\gamma$ is nonorientable means that $\Sigma b_i$ must be odd. 
     
Let $C_i$ denote that union of the edges of the graph between $a_i$ and $a_{i+1}$ on the path.  Note that $C_i$ corresponds to a segment of contiguous entries of the Gauss code bounded by entries $a_i$ and $a_{i+1}$. Call that sequence $C_i’$. So segments $C_1’,\dots,C_n’$ appear disjointly in the Gauss code of $\gamma$, each of length $b_i+2$. Let $x_1, x_2,\dots, x_n$ correspond to the  lengths of the segments in the Gauss code that make up the complement of $C_1’, \dots, C_n’$ in the Gauss code.

Since there are an even number of entries in $G$, note that the sum $\Sigma b_i + \Sigma x_i$ is even. So $\Sigma b_i$ has same parity as $\Sigma x_i$.  Since $\Sigma b_i$ is odd, so is $\Sigma x_i$.

As mentioned previously, the pair of entries labelled $a_i$ separate $G$ into two segments of even length. For each $a_i$, choose one of them, denoted  $R_i$.

Let $b_i$ and $x_j$ be the lengths corresponding to two segments of the Gauss code that are adjacent to one another,  separated by $a_k$. Then each  $R_l$ will either include both of them or neither of them, except for $R_k$, which must include exactly one of them. 
Writing the sum $\Sigma R_i$ in terms of the $a_l$s and $x_l$s,   and the number of entries that appear between them, $b_i$  will  appear  a number of times that has distinct parity from the number of times that $x_j$ appears. Continuing around the Gauss code, using the fact that $b_i$'s alternate with $x_j$'s as we travel around, this implies that in $\Sigma R_j$, the number of times each $a_i$ appears will have the same parity as the number of times each $a_j$ appears and distinct parity from the number of times each $x_i$ appears. 

Since $\Sigma a_i$ is odd, it must be the case each $a_i$ has odd parity. So each $x_i$ has even parity, implying that $\Sigma x_i$ is even, a contradiction to the fact $\Sigma x_i$ is odd.


\end{proof}

%% file: turaevvolume.tex
The purpose of this section is to define invariants of classical and virtual links derived from the tg-hyperbolic metrics associated to the projections of links.

Before doing so, we need to extend the proof given of Theorem \ref{alttg} in \cite{small2017} appropriately to the case of $S$ a sphere, projective plane or Klein bottle. We note that the case of a Klein bottle, although not explicitly included in Lemma 14 of \cite{small2017}, follows immediately by the same argument given there, which is to lift the link to a fully alternating link in $T \times I$, with $T$ a torus, which has been proved to be hyperbolic, and use the fact that the covering translation can be realized by an isometry by the Mostow-Prasad Rigidity Theorem. In the case that $S$ is a sphere, we must exclude an alternating 2-braid knot. But then any other connected prime alternating diagram is hyperbolic by seminal work of Menasco (\cite{Menasco}).

     In the case that $S$ is a projective plane, and $M$ is a twisted I-bundle over $S$, we can take a  double cover $M'$ of $M$ that is $S^2 \times I$, and such that $L$ lifts to a prime fully alternating link $L'$ in $M$. But we do need to exclude the possibility that $L'$ is a 2-braid or that $L'$ is not prime even though $L$ was. In Figure \ref{fig:projective2braid}, we see the only reduced alternating link projections in the projective plane that lift to a 2-braid diagram in the sphere. In the case that $L'$ is not prime, then work of \cite{Menasco} implies that there is a disk $E$ in the projection plane with boundary intersecting the projection at two points, with crossings to both sides of $\partial E$ in $S^2$. Then, since $L$ is prime,  there must have been a simple closed curve in the projective plane that crossed the projection once that lifts to $\partial E$. Excluding those possibilities, it is true that every prime fully alternating link in a thickened surface $M$ is $tg$-hyperbolic.
     
     \begin{figure}[htbp]
    \centering
    \includegraphics[width=0.3\textwidth]{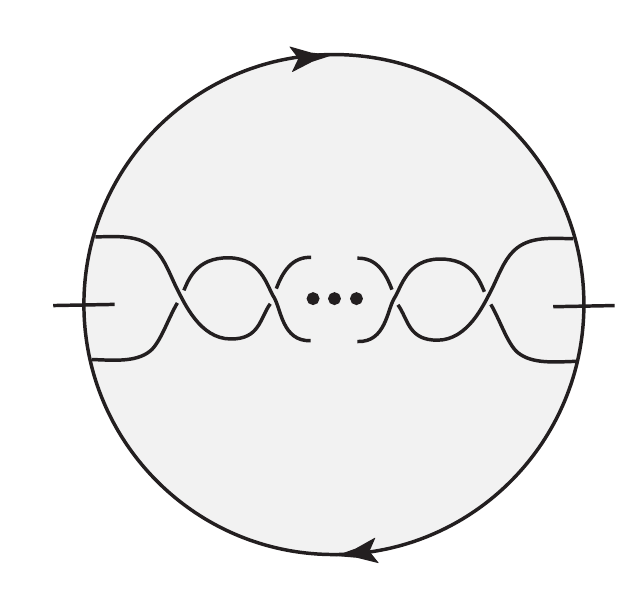}
    \caption{Fully alternating inks in $P^2$ that must be excluded for tg-hyperbolicity.}
    \label{fig:projective2braid}
\end{figure}

We also need several more definitions.

\begin{definition}
Given a (generalized) Gauss code $G$ treated as a cycle, a \emph{subcode} of $G$ is a proper consecutive sequence in the cycle such that if a given number appears in the subcode, its second copy in the Gauss code also appears in the subcode. If $G'$ is a subcode of $G$, we denote this as $G'\subset G$. Note that such subcodes come in pairs, the union of which is the entire code.
\end{definition}

\begin{example}Dropping the letters and signs, 
1234535421 has subcodes 345354 and 2112.
\end{example}

\begin{definition}Given a diagram of a knot, a \emph{virtualization} is the replacement of a crossing by two virtual crossings and the classical crossing as in Figure \ref{fig:virtualization}.(Virtualization is sometimes defined with the central crossing switched but we will mean it as depicted here.) 
\end{definition}

Virtualizations of a given knot have much in common with the original. For instance, in \cite{Kauff}, it was proved that two knots related through a sequence of such virtualizations have the same Kauffman/Jones polynomial.
 
      \begin{figure}[htbp]
    \centering
    \includegraphics[width=0.3\textwidth]{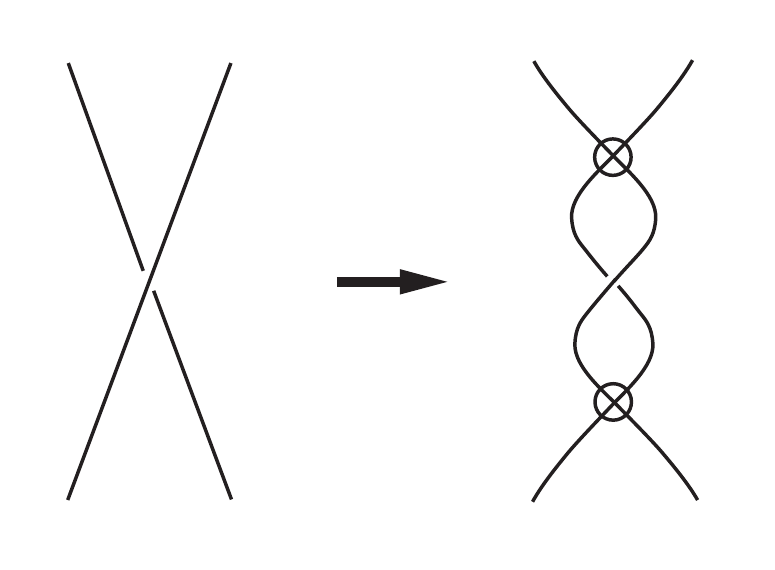}
    \caption{The virtualization of a crossing.}
    \label{fig:virtualization}
\end{figure}

We are now prepared to state the main result of this paper.

\begin{theorem}
Any link $L$ (which may be classical or virtual) has some diagram $D$ such that the corresponding Turaev surface-link pair is tg-hyperbolic. If $L$ is classical, then the diagram $D$ can be taken to be classical. Thus, every classical and virtual link has a Turaev volume and every classical link has a classical Turaev volume.
\end{theorem}

\begin{proof}By our extension of Theorem \ref{alttg} of \cite{small2017}, it is enough to prove that every link has a diagram $D$ such that the corresponding Turaev surface-link pair is prime, connected and fully alternating and is neither a 2-braid in a sphere nor a 2-braid in a projective plane.

    First, note that if the diagram we begin with is disconnected and/or not reduced, we can obtain an equivalent diagram that is connected and reduced by performing a finite set of Reidemeister I and II moves. So we assume from now on that the diagrams are reduced and connected. Then the fact the Turaev surface link pair is fully alternating comes from the construction.
    
    If the diagram is a classical 2-braid diagram, we can perform a Type II 
    Reidemeister move to obtain a new diagram that is not such. If our diagram is a 2-braid with one virtual crossing,  which would yield a projective plane Turaev surface and a 2-braid within it as in Figure \ref{fig:projective2braid}, we can perform a Type II Reidemeister move to to obtain a new diagram that is not such. Note that Theorem 4.2 of \cite{virtualknotsbook} states that a given Turaev surface-link pair can correspond to more than one virtual link, but the corresponding diagrams are related through detour moves and virtualizations. Thus, the only diagrams that yield a Turaev surface-link pair as in  Figure \ref{fig:projective2braid} are a 2-braid with one virtual crossing or a virtualization of it. So in the case of a diagram that is a virtualization of a 2-braid with one virtual crossing, we also apply a Type II Reidemeister move to obtain a new diagram which no longer generates this Turaev surface-link pair. The final case to consider is the trivial knot. If  we change one crossing in the standard projection of the trefoil knot, we obtain a nontrivial projection of the trivial knot, which works.
    
    Now, we have satisfied all of the necessary conditions except primeness. By Theorem \ref{primefullyalt}, this is equivalent to finding a reduced connected diagram such that there is no disk $E$ on the Turaev surface such that $\partial E$ intersects the link twice and contains crossings in its interior. In other words, it is obviously prime.
    
    Suppose that $D$ is a reduced connected diagram for a link $L$ such that the Turaev surface-link pair $\big(S_T(D), L_T(D)\big)$  is not obviously prime. Up to surface isotopy, there is a finite collection of disjoint disks on the surface that make the projection in the Turaev surface not obviously prime. Let $E'$ be one such disk on the Turaev surface. 
    
    Let $G$ denote the Gauss code associated to $D$ and $\tilde{G}$ denote the Gauss code associated to $\big(S_T(D), L_T(D)\big)$. The existence of $E'$ implies that $\tilde{G}$ has a corresponding nontrivial proper subcode obtained as we traverse that portion of $L_T(D)$ inside $E$. Since the Turaev construction preserves subcodes, $G$ must also have a corresponding subcode. That subcode is classical in the sense that it represents a portion of $L_T(D)$ that is a single arc in $L$ together possibly with entire components, that clearly exists in a disk on $S_T(D)$. It therefore corresponds to a portion of $L$ that exists in a disk $E$ in the plane such that it consists of a single arc and possibly a set of entire components. Although there can be other arcs of $L$ intersecting the disk, they can only have virtual crossings with the first portion.
    
    Because $\big(S_T(D), L_T(D)\big)$ is a connected reduced fully alternating diagram, there must be crossings of $L_T(D)$ inside and outside $\partial E'$. This means there are classical crossings of $D$ inside and outside  $\partial E$. Leaving $E$ along one of the two strands $w$ of $D$ that cross $\partial E$, continue along the strand until the first crossing is reached. Let $q$ be the cross-strand. Similarly, following $w$ inside $E$ until it hits the first classical crossing, let $r$ be the cross-strand here. Now do a Type II Reidemeister move of a small piece of $q$ across a small piece of $r$ as in Figure \ref{fig:makingprime}. Note that this destroys the twice-crossed circle that was the boundary of $E$. In the case that $D$ was a classical diagram, the resulting diagram $D'$ is also classical. In the case that $D$ is virtual, note that the Reidemeister Type II move may have to pass over arcs that have virtual crossings with $w$. In this case, make the resultant crossings virtual.
    
    \begin{figure}[htbp]
    \centering
    \includegraphics[width=0.7\textwidth]{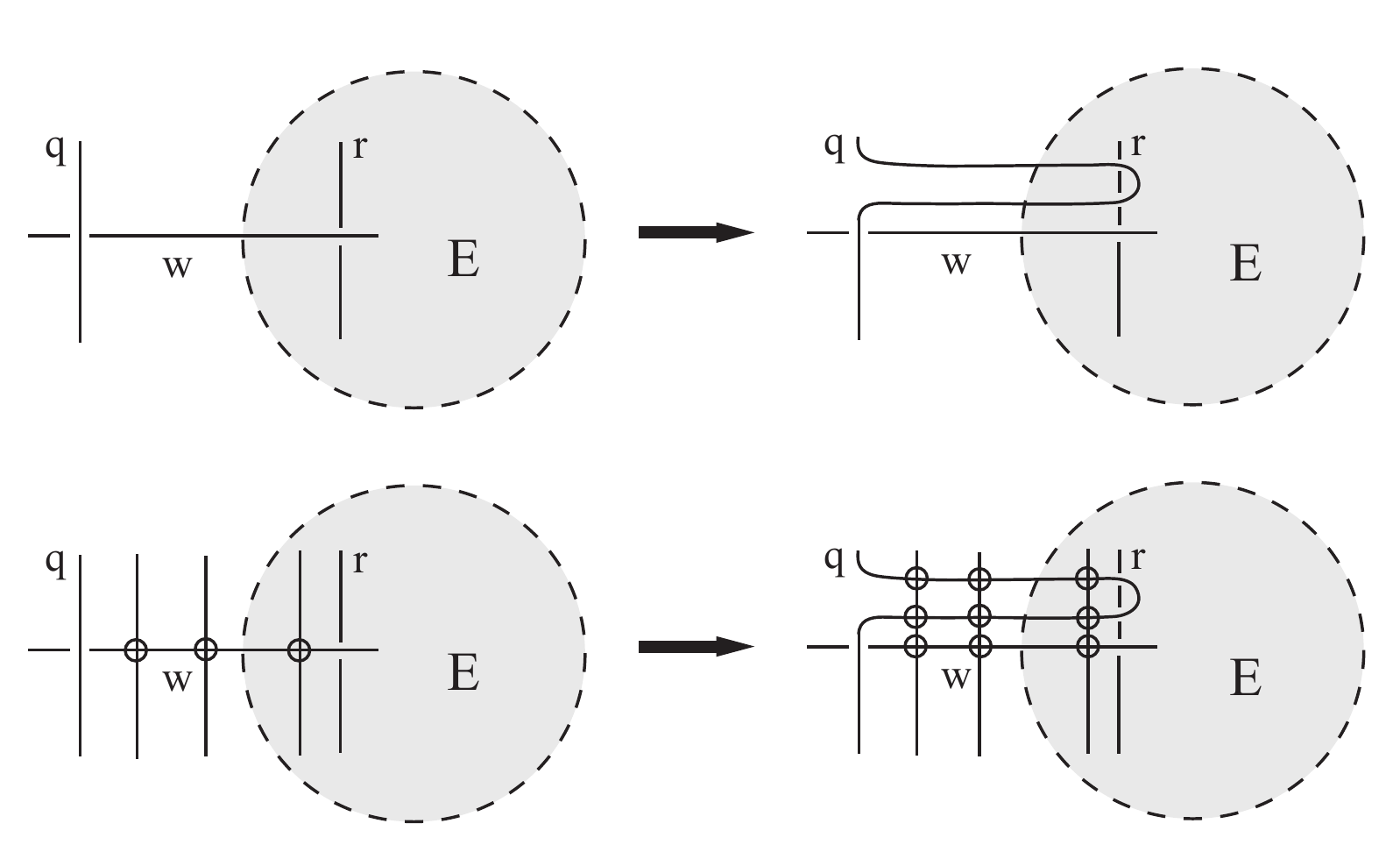}
    \caption{Doing a Type II Reidemeister move to obtain an obviously prime fully alternating Turaev surface-link pair.}
    \label{fig:makingprime}
\end{figure}

The end result is that we have reduced the number of subcodes in $\tilde{G}$. By induction, we obtain a diagram $D''$ of $L$ such that its Turaev surface-link pair is fully alternating, reduced, connected and prime. If the original diagram was classical, the new diagram is classical as well. Hence, the surface-link pair is tg-hyperbolic.

    
    
\end{proof}

We can thus define Turaev volume as in Definition~\ref{def:turaevvolume}. There are several questions one would like to ask.

\begin{question}
    For a nontrivial classical link $L$, does the Turaev volume equal the classical Turaev volume?
\end{question}

\begin{question}
Is the hyperbolic volume of a hyperbolic alternating classical link equal to  the Turaev volume for that link?
\end{question}

This last question seems highly likely, since such a link has a realization with Turaev genus 0 with Turaev volume equal to the hyperbolic volume of the link. For a non-alternating hyperbolic classical link, it seems highly unlikely, since the least volume of a Turaev realization will be for a surface of higher genus than corresponds to its hyperbolic volume in the 3-sphere. 


\begin{lemma}
    For any classical or virtual link $L$, there are infinitely many distinct hyperbolic links in thickened surfaces arising as the Turaev surface of some diagram of $L$.
\end{lemma}

\begin{proof}
Given a diagram for a link, we can compose with the diagram for the trivial knot obtained by changing one crossing in the standard diagram of the trefoil. Each such composition increases the genus of the corresponding Turaev surface.
\end{proof}

\begin{corollary}
    There are infinitely many distinct Turaev volumes associated to any link $L$. 
\end{corollary}

\begin{proof}
As the genus increases, so must the corresponding volumes, as follows. By Miyamoto (\cite{Miyamoto}, paragraph after the proof of Theorem 5.2), we know that the hyperbolic volume of the complement of a link in a thickened orientable surface of genus $g >1$ is at least $(2g-2) v_{oct}$. In the case of a link in a twisted I-bundle over a nonorientable surface of genus $g > 1$, the boundary is a single totally geodesic orientable surface of twice the Euler characteristic, again yielding a volume of at least $(2g -2) v_{oct}$. 
\end{proof}

\noindent We can thus define Turaev spectrum as in Definition~\ref{def:turaevspectrum}.

\begin{theorem}
Let $L$ be a link, classical or virtual. Then $L$ has a non-discrete, well-ordered Turaev spectrum. 
\end{theorem}
\begin{proof}
Since the set of volumes of hyperbolic 3- manifolds are well-ordered, any subset has a least element. So we need only prove the collection is not discrete.

Let $D$ be a diagram of $L$ that lifts to a hyperbolic link in its Turaev surface. Pick any arc $\alpha$ on any strand of $D$. We construct the (equivalent) virtual diagram $D_n$ as follows. Apply $n$ RI moves to $\alpha$. Then, as indicated in Figure~\ref{fig:indiscrete-spectrum-dn}, apply an RII move to $\alpha$ (this ensures that the application of the $n$ RI moves cannot be undone when we lift to the Turaev surface). Note that $D_n$ admits $n+2$ more crossings than $D$, has the same number of closed curves in its $B$-state, and has $n$ more closed curves in its $A$-state (corresponding to the $n$ new crossings added) as in Figure \ref{fig:indiscrete-spectrum-statecurves}. This means that $g_T(D_n) = g_T(D_{n-1})$. The surface-link pairs corresponding to $D_1, D_2, \dots$ have identical diagrams except for the lengthening sequence of bigons. All of the links in the thickened surface can be obtained by Dehn filling an additional trivial component in the thickened surface that wraps around the sequence of bigons. By work of Thurston (see \cite{Thurston}), the sequence of volumes of the sequence must approach the volume of the link with the additional trivial component from below.
\begin{figure}[htbp]
    \centering
    \includegraphics[width=0.4\textwidth]{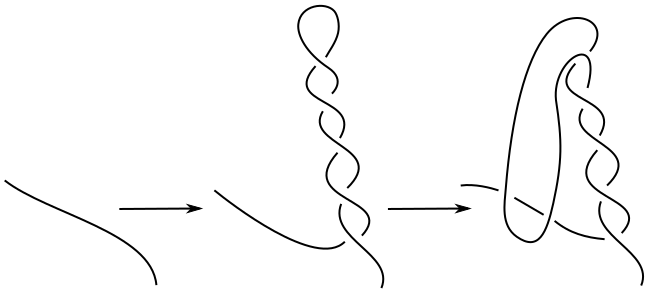}
    \caption{Construction of $D_5$ on some arc $\alpha$}
    \label{fig:indiscrete-spectrum-dn}
\end{figure}

\begin{figure}[htbp]
    \centering
    \includegraphics[width=0.3\textwidth]{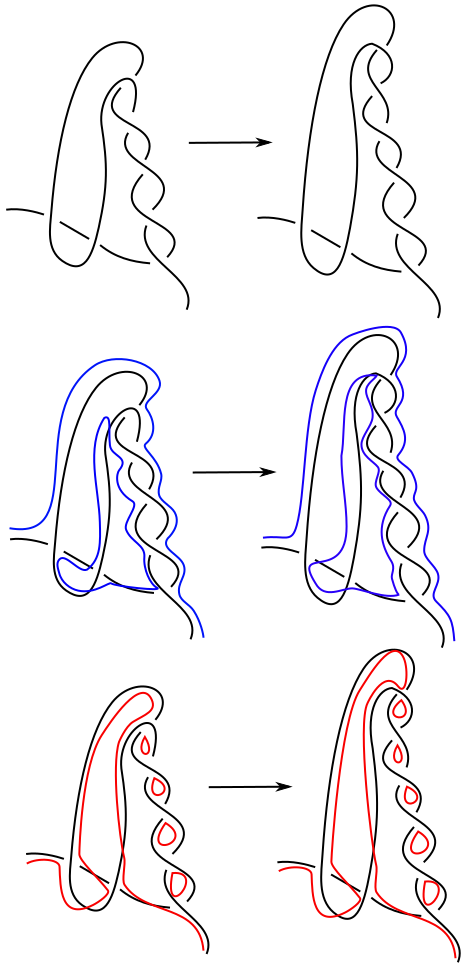}
    \caption{Adding twists preserves the amount of $B$ state circles (blue) and increases the amount of $A$ state circles (red) by one.}
    \label{fig:indiscrete-spectrum-statecurves}
\end{figure}
\end{proof}

\begin{remark}
Note that the Turaev spectrum does not contain the volume to which the volumes in this sequence limit, as the diagram with the trivial component upon which we are doing surgeries does not correspond to the link in question.
\end{remark}

\begin{example}
    Figure \ref{fig:fakeunknottwists} shows the first two in a sequence of unknot diagrams with Turaev volumes that limit toward the volume $14.9004553215$, which corresponds to the volume of the link depicted in  Figure \ref{fig:limitknot}.
   Table~\ref{tab:unknotvolumes} shows the volumes of the first $10$ diagrams in this sequence. 
    
     \begin{figure}[htbp]
        \centering
        \includegraphics[width=0.1\textwidth]{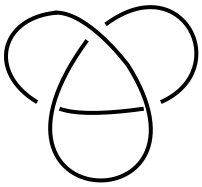}\hspace{0.5cm}\includegraphics[width=0.1\textwidth]{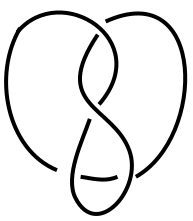}
        \caption{Unknot diagrams with 1 and 2 twists}
        \label{fig:fakeunknottwists}
    \end{figure}
    
    \begin{table}[htbp]
        \centering
        \begin{tabular}{|c|c|}\hline 
            Twists & $\vol_T(D)$\\ \hline
            1 & 9.503403931\\ \hline
            2 & 12.07764776\\ \hline 
            3 & 13.2804231421\\ \hline
            4 & 13.8804968156\\ \hline
            5 & 14.206363788\\ \hline
            6 & 14.399452630\\ \hline
            7 & 14.522417584\\ \hline
            8 & 14.6052962032\\ \hline
            9 & 14.663716611\\ \hline
            10 & 14.7064051972\\ \hline
            \vdots & \vdots \\ \hline
        \end{tabular}
        \caption{Converging unknot volumes}
        \label{tab:convergingunknotvolumes}
    \end{table}
   
   \begin{figure}[htbp]
        \centering
        \includegraphics[scale =.5]{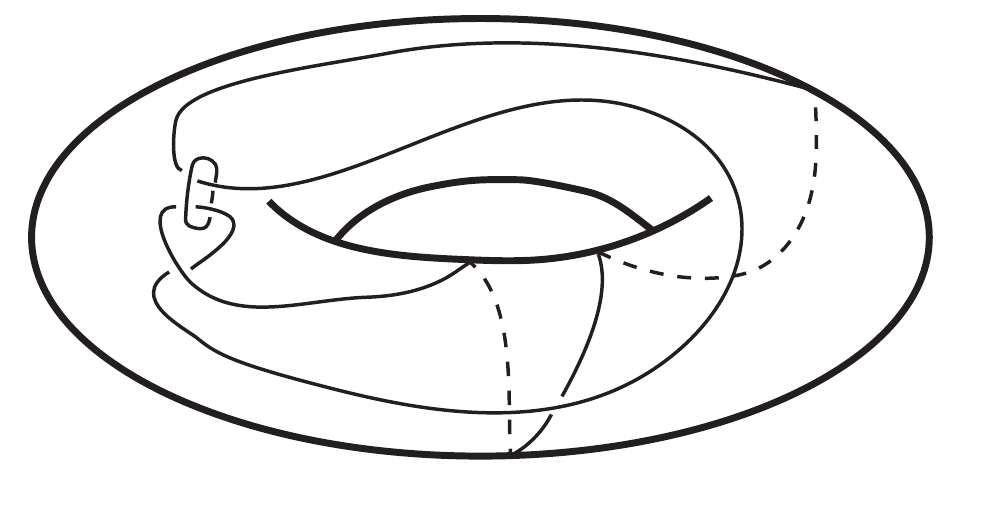}
        \caption{The sequence of diagrams of trivial knots have Turaev volumes limiting to the volume of this link.}
        \label{fig:limitknot}
    \end{figure}
    
\end{example}

\begin{theorem}
The mapping taking virtual knots to Turaev volume is finite-to-one.
\end{theorem}
\begin{proof}

 As mentioned previously, \cite{Miyamoto} implies the volume is at least $(2g-2) v_{oct}$. 
    Hence, for a given volume there are only finitely many genera of surfaces $S$ such that the complement of a link in $M$ can have that volume. Once we have specified the surface, work of Thurston \cite{Thurston} implies that there are at most a finite number of links in the surface with a given volume. Each such link, which is fully alternating, has a finite number of reduced alternating projections. This follows from the fact that any such projection must have the same number of crossings by Theorem 1.1 of \cite{AFLT}.
 Each such projection can only come from a finite number of virtual link diagrams, including a given link diagram, and detour moves and virtualizations of it by Theorem 4.2 of \cite{virtualknotsbook}.

\end{proof}

%% file: somevolumes.tex
We can determine the classical Turaev volume for several types of knots and links. 

\begin{example}
    Any classical hyperbolic alternating knot $K$ with volume less than $v_{oct}$ has classical Turaev volume equal to its hyperbolic volume $vol(K)$. For instance, this holds for the twist knots $4_1,5_2,6_1$, $7_2$,$\ldots$ 

 A classical knot $K$  has orientable Turaev surface for any classical diagram. By \cite{Miyamoto}, the volume of the complement of a link in a thickened orientable genus $g\geq 2$ surface is at least $(2g-2)v_{oct}$ and since $g\geq 2$, this is at least $2v_{oct}$.  For any genus one Turaev surface $S_T(D)$, the volume of $S_T(D) \times I \setminus L_T(D)$ is the volume of $S^3\setminus (H \cup L_T(D))$ where $H$ is a Hopf link. Since $H\cup L_T(D)$ is a $3$-component link, the volume is at least $v_{oct}$. This is true because in \cite{Agol}, Agol proved that the smallest orientable 2-cusped hyperbolic manifold has volume $v_{oct}$. If a 3-cusped manifold had volume less than this, high surgery on one cusp would also have volume less than $v_{oct}$, contradicting Agol's result.  Thus to obtain a volume less than $v_{oct}$, the Turaev surface must be a sphere. A reduced classical non-alternating diagram of $K$ yields a higher genus Turaev surface. So the only possible classical diagram that yields a spherical Turaev surface is an alternating diagram, all of which yield $K$ as the knot on the spherical Turaev surface with volume $vol(K)$.
 
 Note that although there is a 12-crossing knot (12n0242) that has the same volume as the $5_2$ knot, which is $2.828\dots$, it is not alternating, and therefore we do not expect its Turaev volume to be $2.828\dots$, since it will not have a minimal genus Turaev surface of genus 0.
 
 \end{example}
 
\begin{example}
    For classical alternating two component links with classical hyperbolic volumes less than $2v_{oct}$, the classical Turaev volume equals the hyperbolic volume. For instance, this includes the links $5_1^2$, $6_2^2$, $6_3^2$,$7_6^2$, $7_1^2$,$7_2^2$,$7_3^2$,$7_4^2$, and $8_2^2$.

To see this, let  $L$ be a classical alternating two-component link with vol$(S^3\backslash L) < 2v_{oct}$. 
 For classical diagrams with Turaev genus at least $2$, the volume of $S_T(D)
 \times I \setminus L_T(D)$ is again at least $2v_{oct}$ by \cite{Miyamoto}.

Furthermore, in \cite{Yoshida}, it was proved that the smallest volume of a $4$-cusped orientable hyperbolic 3-manifold is $2v_{oct}$. If some classical diagram of $L$ lifts to $L'$ on an orientable genus one Turaev surface $T$, the volume of $T\backslash L'$ equals the volume of $S^3 \setminus (L'\cup H)$ where $H$ is an appropriately linked Hopf link, and since this will have four cusps,  its volume is at least $2v_{oct}$. 

Then, $L$ can only have classical Turaev volume less than $2v_{oct}$ in its genus $0$ Turaev surfaces. To have a corresponding genus 0 Turaev surface, the diagram must be alternating, and the Turaev surface-link pair, after capping off the two spherical boundaries is in fact $L$ in $S^3$. 

Note however, that if we want to calculate the Turaev volume, we must consider virtual diagrams of $L$ as well. Although we can eliminate any nonorientable Turaev surfaces of genus at least 2 because the volume will be too large, we cannot eliminate the projective plane, the Klein bottle or the nonorientabe surface of genus 3/2 and prove that there is not a smaller volume Turaev surface-link pair. 

\end{example}

\begin{example}
We conjecture that the trivial knot has classical Turaev volume equal to $9.5034 \dots$, corresponding to a projection that is the standard projection of the trefoil knot with one crossing switched. However, this cannot be its Turaev volume, as we can take the projection obtained from the trivial projection by doing a virtual Type I  Reidemeister move followed by a classical Type II Reidemeister move, as in Figure \ref{fig:trivialvirtual}. The resulting projection does yield a Turaev surface-link pair that is a hyperbolic knot in a Klein bottle with volume $2v_{oct} = 7.3277247535 \dots$. We conjecture that this is its Turaev volume.

Note that the standard figure-eight knot projection with any two crossings that do not share a bigon made virtual shares this same Turaev-surface link pair, as it can be obtained from this projection of the trivial knot by virtualization and detour moves. So we conjecture that it also has Turaev volume $2v_{oct}$.

 \begin{figure}[htbp]
    \centering
    \includegraphics[width=1.0\textwidth]{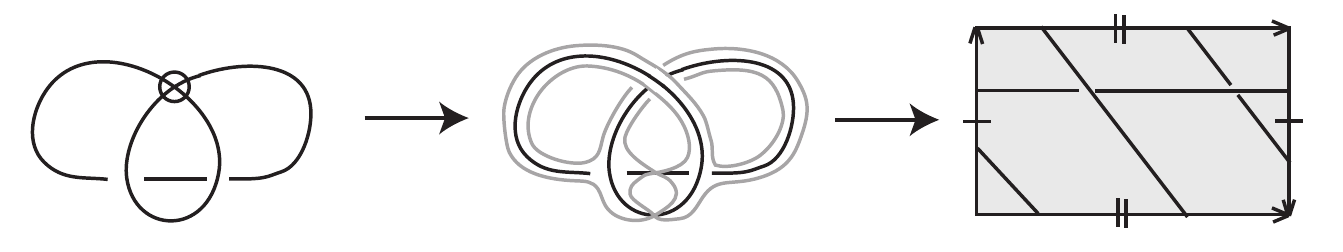}
    \caption{Finding the Turaev surface-link pair for this virtual projection of the trivial knot.}
    \label{fig:trivialvirtual}
\end{figure}

The trivial link of two components has a projection obtained by doing a single classical Type II Reidemeister move of one component over the other. The resulting Turaev surface-link pair is a torus with a link projection of two crossings upon it. It is hyperbolic with a volume of $6.089\dots$, less than the conjectured Turaev volume of the trivial knot.

The virtual figure-eight knot (obtained by making one crossing of the classical figure-eight knot projection virtual) is a knot of Turaev genus 1/2 with Turaev volume $2.66674478345\dots$ for this projection. We conjecture this is its Turaev volume. This is the lowest Turaev volume for a non-classical virtual knot yet discovered. (The standard projection of the virtual trefoil yields a non-hyperbolic Turaev surface-link pair.)


\end{example}

\begin{example} 
In Table 3, we see three classical projections of the trefoil knot, which itself is not hyperbolic. The least volume occurs for the first projection, and we conjecture that the corresponding volume is the classical Turaev volume for the trefoil. 

We expect this is also the Turaev volume, since for simple examples of projections of the trefoil that include virtual crossings, the volume is much larger.
\end{example}

%% file: unknot-diagrams.tex
\begin{table}[!t]
    \centering
    \begin{tabular}{|c|c|c|}
        \hline
        Unknot Diagram &  Turaev Genus & Volume \\
        
        \hline
        \begin{minipage}{.15\textwidth}
        \includegraphics[width=\linewidth]{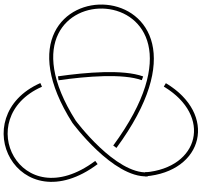}
        \end{minipage} & $g_t = 1$ & $9.503403931$ \\
        
        \hline
        \begin{minipage}{.15\textwidth}
        \includegraphics[width=\linewidth]{figures/unknotted-8.png}
        \end{minipage} & $g_t = 1$ & $12.07764776$ \\ 
        
        \hline
        \begin{minipage}{.15\textwidth}
        \includegraphics[width=\linewidth]{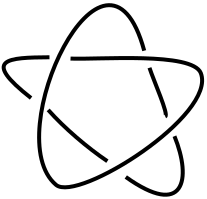}
        \end{minipage} & $g_t = 1$ & $16.8804404$ \\ 
        
        \hline
        \begin{minipage}{.15\textwidth}
        \includegraphics[width=\linewidth]{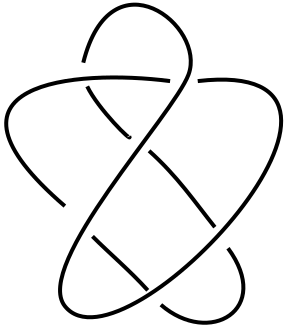}
        \end{minipage} & $g_t = 2$ & $33.595745176$\\ 
        
        \hline
        \begin{minipage}{.15\textwidth}
        \includegraphics[width=\linewidth]{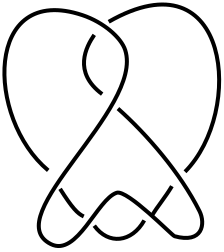}
        \end{minipage} & $g_t = 2$ & $35.488291197$ \\ 
        
        \hline
        
    \end{tabular}
    \caption{Turaev volumes of various classical diagrams of the unknot with Turaev genus $g_t\leq 2$.}
    \label{tab:unknotvolumes}
\end{table}

\clearpage

\begin{table}[!t]
    \centering
    \begin{tabular}{|c|c|c|}
        \hline
        Trefoil Diagram &  Turaev Genus & Volume \\
        
        \hline
        \begin{minipage}{.15\textwidth}
        \includegraphics[width=\linewidth]{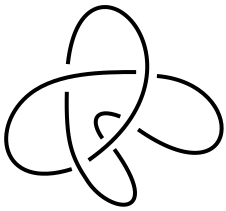}
        \end{minipage} & $g_t = 1$ & $11.3328915634$\\
        
        \hline
        \begin{minipage}{.15\textwidth}
        \includegraphics[width=\linewidth]{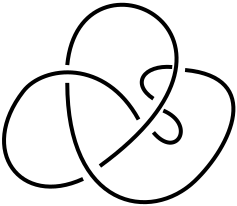}
        \end{minipage} & $g_t = 1$ & $13.541527117$\\
        
        \hline
        \begin{minipage}{.15\textwidth}
        \includegraphics[width=\linewidth]{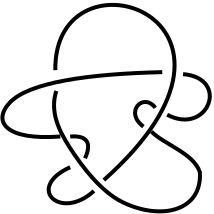}
        \end{minipage} & $g_t = 1$ & $20.179363683$\\
        
        \hline
        
    \end{tabular}
    \caption{Turaev volumes of various classical diagrams of the trefoil knot.}
    \label{tab:trefoilvolumes}
\end{table}